\documentclass{amsart}[12pt]
\usepackage{amsmath,amsfonts,amssymb,amscd}

\DeclareMathSymbol{\twoheadrightarrow}  {\mathrel}{AMSa}{"10}

\def\Q{{\mathbb Q}}
\def\Z{{\mathbb Z}}
\def\C{{\mathbb C}}
                     \def\YY{{\mathbb Y}}

\def\F{{\mathbb F}}

\def\Sn{{\mathbf S}_n}
\def\An{{\mathbf A}_n}
\def\RR{{\mathfrak R}}

                                  \def\KK{{\mathbf K}}
\def\Perm{\mathrm{Perm}}
\def\Gal{\mathrm{Gal}}

\def\End{\mathrm{End}}
\def\Aut{\mathrm{Aut}}
\def\Alt{\mathrm{Alt}}

                              \def\Spec{\mathrm{Spec}}
                               \def\spec{\mathrm{spec}}

\def\I{{\mathcal I}}

                                     \def\XX{{\mathcal X}}

\def\ST{{\mathbf S}}

\def\fchar{\mathrm{char}}

\def\dim{\mathrm{dim}}

\def\OC{{\mathcal O}}

\def\a{{\mathfrak a}}
\def\b{{\mathfrak b}}

\def\mm{{\mathfrak m}}

\newtheorem{thm}{Theorem}[section]
\newtheorem{lem}[thm]{Lemma}

\theoremstyle{definition}

\newtheorem{ex}[thm]{Example}
\newtheorem{rem}[thm]{Remark}

\newtheorem{rems}[thm]{Remarks}
\newtheorem{sect}[thm]{}
\title[Endomorphisms of ordinary superelliptic Jacobians]
{Endomorphisms of   ordinary superelliptic Jacobians}
\author[Yuri\ G.\ Zarhin]{Yuri\ G.\ Zarhin}

\address{Department of Mathematics, Pennsylvania State University,
University Park, PA 16802, USA}

 \email{zarhin\char`\@math.psu.edu}
 \thanks{Partially supported by Simons Foundation Collaboration grant $\# 585711$. Part of this work was done during my stay at the Institute Henri Poincar\'e (June - July 2019) and the Weizmann Institute of Science (December 2019 - January 2020),
 whose hospitality and support are gratefully acknowledged.}
 
 \dedicatory {To Gerhard Frey}

\begin{document}

\begin{abstract}
  Let $K$ be a field of prime characteristic $p$, $n\ge 5$ an integer, $f(x)$ an irreducible
polynomial over $K$ of degree $n$, whose Galois group is either the full symmetric group $\ST_n$ or the alternating group $\An$.
Let $\ell$ be an odd
prime different from $p$, $\Z[\zeta_{\ell}]$ the ring of integers in
the $\ell$th cyclotomic field,
 $C_{f,\ell}:y^{\ell}=f(x)$  the  corresponding superelliptic curve and
$J(C_{f,\ell})$ its Jacobian.  We prove that the ring of all
$\bar{K}$-endomorphisms of $J(C_{f,\ell})$ coincides with $\Z[\zeta_{\ell}]$
if $J(C_{f,\ell})$ is an {\sl ordinary} abelian variety and $(\ell,n)\ne (5,5)$.
\end{abstract}

\subjclass[2010]{14H40, 14K05, 11G30, 11G10}

\keywords{\sl ordinary abelian varieties,
superelliptic Jacobians, endomorphisms of abelian varieties.}

\maketitle

\section{Introduction}
 Throughout this paper $K$ is a field and
$\bar{K}$  is its algebraic closure. We write $\Gal(K)$ for the absolute
Galois group $\Gal(\bar{K}/K):=\Aut(\bar{K}/K)$ of $K$. If $X$ is an abelian variety of positive
dimension over $\bar{K}$, then $\End(X)$ stands for the ring of all its
$\bar{K}$-endomorphisms and $\End^0(X)$ for the corresponding
$\Q$-algebra $\End(X)\otimes\Q$. 
 We write $1_X$ for the identity
endomorphism of $X$. If $X$ is defined over $K$ then we write $\End_K(X)$ for the ring of all $K$-endomorphisms of $X$.
We have
$$\Z=\Z\cdot 1_X \subset \End_K(X) \subset \End(X),$$ 
$$\Q=\Q\cdot 1_X \subset \End_K(X)\otimes\Q=:\End_K^{0}X \subset \End^{0}(X).$$

Let $\C$ be the field of complex numbers,  $q$ a positive integer,
$\zeta_{q} \in \C$ a primitive $q$th root of unity,
$\Q(\zeta_{q})\subset \C$ the $q$th cyclotomic field, and
$\Z[\zeta_{q}]$ the ring of integers in $\Q(\zeta_{q})$.

The aim of this paper is to construct in prime characteristic $p$  explicit examples of abelian varieties $X$
with  $\End(X)\cong \Z[\zeta_{\ell}]$ where $\ell$ is a prime different from $p$. In characteristic 0
such examples are provided by the Jacobians $J(C_{f,\ell})$ of superelliptic curves $C_{f,\ell}:y^{\ell}=f(x)$ where $f(x) \in K[x]$
is an irreducible polynomial with ``large'' Galois group \cite{ZarhinMRL,ZarhinCrelle,ZarhinCamb}. Namely, if $\fchar(K)=0$, 
$\deg(f) \ge 5$ and the Galois group $\Gal(f)$ of $f(x)$ over $K$ is either the full symmetric group $\ST_n$ or the alternating group $\An$
then $\End(J(C_{f,\ell})\cong \Z[\zeta_{\ell}]$. 
In the present paper we prove (Theorem \ref{mainordinary})) that this assertion about the endomorphism rings remains true in characteristic $p$
under   additional assumptions that  $\ell$ is odd, $J(C_{f,\ell})$ is an {\sl ordinary} abelian variety  and $(\ell,n)\ne (5,5)$.
(The hyperelliptic case   $\ell=2$ was treated earlier in \cite{ZarhinBSMF}.) We provide explicit examples of such polynomials
(Example \ref{exOrdinary}). We also discuss the case of the Jacobians of superelliptic  $y^q=f(x)$ where $q$ is a power of $\ell$
(Theorem (\ref{mainordinaryq}).

The paper is organized as follows.  Section \ref{one} contains the necessary background from the theory of finite
permutation groups,  their natural (permutational) linear representations in characteristic $\ell$, superelliptic curves and their Jacobians.

In Section \ref{PlanO} we outline the plan of the proof of our main results (Theorems \ref{mainordinary} and \ref{mainordinaryq}). This section also contains certain auxiliary results 
(Theorems \ref{generalA} and \ref{ordinary}) about the endomorphism rings of abelian varieties that do not have to be superelliptic Jacobians; these results
 may be of certain independent interest. In the same section we deduce Theorem \ref{mainordinary} from Theorems \ref{generalA} and \ref{ordinary}.
We prove Theorems \ref{generalA}  in Section \ref{pfGA}. Theorem   \ref{ordinary} is proven in Section \ref{Pordinary}.
In order to prove Theorem \ref{generalA}, we need to use an information about  endomorphism rings of abelian varieties and
their torsion points that is discussed in Section \ref{endomS}.
% and \ref{abmult}. 
In order to deduce 
Theorem  \ref{ordinary} from Theorem \ref{generalA}, we use properties of abelian schemes (Section \ref{Aschemes}), including
 canonical liftings of ordinary abelian varieties to characteristic 0 (Remark \ref{ordinarylift}). Theorem \ref{mainordinaryq} is proven 
in Section \ref{superQ}.

{\bf Acknowledgements}.
I am grateful to Noriko Yui and Daqing Wan for stimulating discussions. My special thanks go to Tatiana Bandman and the referees,
whose comments helped to improve the exposition.

\section{Definitions and Statements}
\label{one} 
\begin{sect}{\bf Polynomials, Galois groups and Permutational Representations.} Let $\ell$ be a prime such that $\fchar(K)\ne \ell$. Let $f(x) \in K[x]$ be a separable polynomial
of degree $n\ge 3$,
 $\RR_f\subset \bar{K}$ the ($n$-element) set of roots of $f$ and $K(\RR_f)\subset \bar{K}$
the splitting field of $f$. We write $\Gal(f)=\Gal(f/K)$ for the
Galois group $\Gal(K(\RR_f)/K)$ of $f$; it permutes the roots of
$f$ and may be viewed as a certain permutation group of $\RR_f$,
i.e., as  a subgroup of the group $\Perm(\RR_f)\cong\Sn$ of all
permutations of $\RR_f$. We write $\Alt(\RR_f)\cong \An$ for the
only subgroup of index 2 in $\Perm(\RR_f)$ that consists of all even permutations of $\RR_f$. Slightly abusing notation, we say that
$\Gal(f)$ is the full symmetric group $\Sn$ (resp. the alternating group $\An$) if $\Gal(f)=\Perm(\RR_f)$ (resp.  $\Alt(\RR_f)$).

Let us consider the standard faithful permutational representation of  the group of permutation $\Gal(f)\subset \Perm(\RR_f)$ in the $n$-dimensional
$\F_{\ell}$-vector space
$$\F_{\ell}^{\RR_f}=\{\phi: \RR_f \to \F_{\ell}\}$$
of $\F_{\ell}$-valued functions on $\RR_f$ and its $(n-1)$-dimensional subrepresesentation in the subspace
\begin{equation}
\label{heart0}
\left(\F_{\ell}^{\RR_f}\right)^{0}:=\{\phi: \RR_f \to \F_{\ell}\mid \sum_{\alpha\in \RR_f}\phi(\alpha)=0\}
\end{equation}
(see \cite{Mortimer}). 
The natural surjection
$$\Gal(K) \twoheadrightarrow \Gal(K(\RR_f)/K)=\Gal(f)$$
provides $\F_{\ell}^{\RR_f}$ with the natural structure of $\Gal(K)$-module
and $\left(\F_{\ell}^{\RR_f}\right)^{0}$ becomes its $\Gal(K)$-submodule.
\end{sect}

\begin{sect}{\bf Superelliptic curves and Jacobians}.
 Let $C_{f,\ell}$ be a smooth projective model of the smooth affine $K$-curve
$y^{\ell}=f(x)$; its genus $g$ is equal to $(n-1)(\ell-1)/2$ if
$\ell$ does {\sl not} divide $n$ and to $(n-2)(\ell-1)/2$ if $\ell$ divides $n$. 
The
Jacobian $J(C_{f,\ell})$ of $C_{f,\ell}$ is a $g$-dimensional abelian variety that
is defined over $K$. 

Suppose that $K$ contains a primitive $\ell$th root of unity, say  $\zeta$. 
The map
 $(x,y) \mapsto (x, \zeta y)$
gives rise to a non-trivial birational $K$-automorphism
$\delta_{\ell}: C_{f,\ell} \to C_{f,\ell}$ of period $\ell$. 
 By Albanese functoriality, $\delta_{\ell}$
induces an automorphism of $J(C_{f,\ell})$ which we still denote by
$\delta_{\ell}$. It is known \cite[p.~149]{Poonen},
\cite[p.~458]{SPoonen}) (see also \cite{ZarhinM,ZarhinPisa}) that
$\delta_{\ell}$ satisfies the $\ell$th cyclotomic equation, i.e.,
${\delta_{\ell}}^{\ell-1}+\cdots +\delta_{\ell}+1=0$, 
 in $\End(J(C_{f,\ell}))$. This gives rise to the ring embeddings
 $$\Z[\zeta_{\ell}] \hookrightarrow \End_K(J(C_{f,\ell}))\subset \End(J(C_{f,\ell})),\
 \zeta_{\ell}\mapsto \delta_{\ell},$$
$$\Q(\zeta_{\ell}) \hookrightarrow \End^0_K(J(C_{f,\ell}))\subset \End^0(J(C_{f,\ell})),\
 \zeta_{\ell}\mapsto \delta_{\ell},$$
which send $1$ to the identity automorphism of $J(C_{f,\ell})$. In
particular,  if $\Q[\delta_{\ell}]$ is the $\Q$-subalgebra of
$\End^0(J(C_{f,\ell}))$ generated by $\delta_{\ell}$ then the latter
embedding establishes a canonical isomorphism of $\Q$-algebras
$$\Q(\zeta_{\ell})\cong \Q[\delta_{\ell}]$$
that sends $\zeta_{\ell}$ to $\delta_{\ell}$.

The set of fixed points of $\delta_{\ell}$
$$J(C_{f,\ell})^{\delta_{\ell}}:=\{z \in J(C_{f,\ell})(\bar{K})\mid \delta_{\ell}(z)=z\}$$
is a $\Gal(K)$-submodule  of $J(C_{f,\ell})(\bar{K})$.
\end{sect}

\begin{lem}[See \cite{Poonen,SPoonen}]
\label{fixedJ}
If $\ell$ does not divide $n$ then the the $\Gal(K)$-module $J(C_{f,\ell})^{\delta_{\ell}}$
is isomorphic to  $\left(\F_{\ell}^{\RR_f}\right)^{0}$ defined by Eq. \eqref{heart0}.
\end{lem}
\begin{sect}{\bf Differentials of the first kind}.
\label{multOmega}
Let $\Omega^{1}(J(C_{f,\ell}))$ be the $g$-dimensional $\bar{K}$-vector
space of (invariant) differentials of the first kind on
$J(C_{f,\ell})$ and
$${\delta_{\ell}}^{*}:
\Omega^{1}(J(C_{f,\ell}))\to \Omega^{1}(J(C_{f,\ell}))$$ the linear
operator induced by functoriality by $\delta_{\ell}$. Clearly,
$({\delta_{\ell}^{*}})^{\ell}$ is the identity map. 
Since
$\fchar(\bar{K})\ne \ell$, the linear operator
${\delta_{\ell}}^{*}$ is diagonalizable. It is known \cite[Remarks
4.5, 4.6 and 4.7 on pp. 352--353]{ZarhinM} that the spectrum of
$\delta_{\ell}^{*}$ consists of (primitive) $\ell$th roots of unity
$\zeta^{-i}$ where $i<\ell$ is a positive integer such that
$[ni/\ell]>0$ and the multiplicity of $\zeta^{-i}$ equals
$[ni/\ell]$. In particular, $1$ is {\sl not} an eigenvalue of
${\delta_{\ell}}^{*}$.
\end{sect}

The structure of  the endomorphism algebra of superelliptic Jacobians was studied in
\cite{ZarhinCrelle,ZarhinCamb,ZarhinSb,
ZarhinM,ZarhinPisa,ZarhinMZ2, ZarhinMZ2arXiv,Xue,XueYu,ZarhinGanita}.
 In particular, the author proved in \cite{ZarhinMRL,ZarhinCamb} that if $\fchar(K)=0, \ n \ge
5$, and $\Gal(f)$ is either  $\Sn$
or  $\An$ then
$\End(J(C_{f,\ell}))=\Z[\zeta_{\ell}]$. 

The  aim of this paper is to extend these results to the case of prime
characteristic under an additional assumption that
$J(C_{f,\ell})$ is an {\sl ordinary} abelian variety. Our main result is the following assertion.

 \begin{thm}
 \label{mainordinary}
Suppose that $K$ is a field of prime characteristic $p$. Let $n\ge 5$ be an integer and   $\ell$ be an odd prime that does not coincide with $p$.
Suppose that 
%$K$ contains a primitive $\ell$th root of unity, say  $\zeta_{\ell}$ and
  $\Gal(f)$ is either $\Sn$
or  $\An$.
 Assume also that $J(C_{f,\ell})$ is an ordinary abelian variety.

If $(\ell,n) \ne (5,5)$ then
 $$\End^0(J(C_{f,\ell}))=\Q[\delta_{\ell}]\cong\Q(\zeta_{\ell}) \ \text{ and } \ \End(J(C_{f,\ell}))=\Z[\delta_{\ell}]\cong\Z[\zeta_{\ell}].$$
\end{thm}

\begin{sect}{\bf Automorphisms and Endomorphisms of Superelliptic Jacobians}.
\label{superqdelta}
Now let us assume that $q=\ell^r$ is a power of a prime $\ell \ne \fchar(K)$, and consider the  smooth projective model $C_{f,q}$ of the affine smooth curve $y^q=f(x)$. If $K$ contains a primitive $q$th root of unity, say  $\zeta$, then $C_{f,q}$ admits a periodic $K$-automorphism
\begin{equation}
\label{deltaq}
\delta_q: C_{f,q} \to C_{f,q}, \ (x,y)\mapsto (x,\zeta y).
\end{equation}
It induces, by Albanese functoriality, the periodic automorphism of $J( C_{f,q})$, which we continue denote by 
$$\delta_q\in \Aut_K(J( C_{f,q}).$$
It is known \cite[Lemma 4.8 on p.354]{ZarhinM} that
\begin{equation}
\label{deltaqE}
\delta_q^q=1_{J( C_{f,q})}, \ \sum_{i=0}^{\ell-1}\delta_q^{i \ell^{r-1}}=0.
\end{equation}
In addition, the $\Q$-subalgebra $\Q[\delta_q]$ of $\End_K^0(J( C_{f,q}))$ generated by $\delta_q$, is canonically isomorphic to the direct sum
 $\prod_{j=1}^r \Q(\zeta_{\ell^j})$ of cyclotomic fields.
\end{sect}

The following assertion was proven by the author in \cite[Th. 1.1 on p. 340]{ZarhinM}.
\begin{thm}
\label{qCC}
Let $K$ be a subfield of $\C$. Suppose that $n\ge 5$ is an integer,   $\ell$ is a prime and $r$ is a positive integer. Let us assume that either $\ell\nmid n$ or  $q:=\ell^r$ divides $n$. 
%Let us consider the smooth projective model $C_{f,q}$ of the affine smooth curve $y^q=f(x)$. 
Suppose that 
  $\Gal(f)$ is either $\Sn$
or  $\An$.
If $(q,n) \ne (5,5)$ then the semisimple $\Q$-algebra
 $\End^0(J(C_{f,q}))$ coincides with $\Q[\delta_q]$ and is  isomorphic to the direct sum
 $\prod_{j=1}^r \Q(\zeta_{\ell^j})$ of cyclotomic fields. In particular, $\End^0(J(C_{f,q}))$  is commutative.
\end{thm}
\begin{rem}
J. Xue \cite{Xue} extended the result of Theorem \ref{qCC} to the remaining case when $\ell\mid n$ but $q$ does {\sl not} divide $n$.
\end{rem}

The second main result of this paper is an analogue of Theorem \ref{qCC} in finite characteristic for ordinary Jacobians.

 \begin{thm}
 \label{mainordinaryq}
Suppose that $K$ is a field of prime characteristic $p$. Let $n\ge 5$ be an integer,   $\ell$ is an odd prime that does not coincide with $p$, and $r$ is a positive integer. Let us assume that either $\ell\nmid n$ or  $q:=\ell^r$ divides $n$. 
%Let us consider the smooth projective model $C_{f,q}$ of the affine smooth curve $y^q=f(x)$. 
Suppose that 
%$K$ contains a primitive $\ell$th root of unity and
  $\Gal(f)$ is either $\Sn$
or  $\An$.
 Assume also that the Jacobian $J(C_{f,q})$ of $C_{f,q}$
 is an ordinary abelian variety.

If $(q,n) \ne (5,5)$ then the semisimple $\Q$-algebra
 $\End^0(J(C_{f,\ell}))$ coincides with $\Q[\delta_q]$ and is  isomorphic to the direct sum
 $\prod_{j=1}^r \Q(\zeta_{\ell^j})$ of cyclotomic fields.
\end{thm}

\begin{rems}
\label{simplify}
\begin{itemize}

\item[(i)]
Replacing $K$ by its {\sl perfectization}  $\sqrt[p^{\infty}]{K}$, we may and will assume in
the course of the proof of Theorems \ref{mainordinary} and \ref{mainordinaryq} that $K$ is a {\sl perfect}
field.
\item[(ii)]
Replacing $K$ by a suitable finite abelian extension,
we may and will assume in the course of the proof of Theorems
\ref{mainordinary} and \ref{mainordinaryq}  that $\Gal(f)=\An$
 and $K$ contains a {\sl primitive} $\ell$th  (respectively $q$th) root of unity.
\item[(iii)] Using an elementary substitution (see \cite[Remark 4.3]{ZarhinM}), we may and will assume in the course of the proof of Theorems
\ref{mainordinary}  and \ref{mainordinaryq} that $\ell$ does {\sl not} divide $n$.
\end{itemize}
\end{rems}

\begin{ex}
\label{exOrdinary}
Let us choose  positive integers $q\ge 2$ and $d$, and assume that
either $q\ge 7$ or $d\ge 2$.  Pick an odd prime $p$ that is
congruent to $1$ modulo $q(dq-1)$ and put $n=dq$; clearly, $n
\ge 6$. Let $K=\bar{\F}_{p}(t)$ be the field of rational functions in
one variable over an algebraic closure $\bar{\F}_{p}$ of the (finite)
prime field $\F_{p}$ of characteristic $p$.  Since $p>n$, the product
$n(n-1)$ is {\sl not} divisible by $p$. This implies that the polynomial
$f(x)=x^n-x-t \in K[x]$ is a {\sl Morse function} in a sense of
\cite[p. 39]{SerreGalois}; in particular, $\Gal(f)=\ST_n$ \cite[Th.
4.4.5]{SerreGalois}. Then the Jacobian $J$ of the
superelliptic  $K$-curve $C: y^q=x^n-x-t$ is an {\sl ordinary} abelian
variety over $K$. In order to prove that, it suffices to check that the
Jacobian $J_0$ of its specialization $$C_0: y^q=x^n-x=x^{dq}-x$$ (at
$t=0$) is an ordinary abelian variety over $\bar{\F}_{p}$. Dividing
the equation of $C_0$ by $x^n$ and using a substitution
$$u=\frac{y}{x^d}, \ v =\frac{1}{x},$$
we obtain that $C_0$ is birationally isomorphic to the curve
$$C^{\prime}: u^q=1-v^{dq-1}.$$
Clearly, $C^{\prime}$ is covered by  Fermat curve
$$w^{m}=1-z^{m}$$
with $m:=q(dq-1)$.
 Since $p-1$ is divisible by $q(dq-1)=m$, the Jacobian of the
Fermat curve is ordinary \cite[Prop. 4.1 and Th. 4.2]{Yui}. Since
$J_0$ is isomorphic to a quotient of the Jacobian of the Fermat curve,
it is also ordinary. This implies that $J$ is ordinary. Now if
$q=\ell$ is an odd prime then $n=dp \ge 2\ell\ge 6$ and the
equality  $\End(J(C_{f,\ell}))=\Z[\zeta_{\ell}]$  follows from
Theorem \ref{mainordinary}. (Compare with \cite{ZarhinH} where the case of
$\ell=3, n=9$ is discussed.) 

 Similarly, if $r\ge 2$ is an integer and  $q=\ell^r$ is a power of an odd prime $\ell$
then it follows from Theorem \ref{mainordinaryq} that
$$\End^0(J(C_{f,\ell}))\cong \oplus_{j=1}^r \Q(\zeta_{\ell^i}).$$
\end{ex}

\section{The plan of the proof}
\label{PlanO}

Our proof is based on {\sl Serre-Tate canonical lifting of ordinary abelian varieties}  \cite{Messing} and the following assertions.

\begin{thm}
\label{generalA}
Suppose that $\ell$ is a prime, $r$ a positive integer, $n\ge 5$ a positive integer that is not divisible by $\ell$,
$q=\ell^r$, and
 $\KK$ be a field of characteristic zero that contains a primitive $q$th root of unity, say  $\zeta$. Let $\YY$ be a positive-dimensional abelian variety over $\KK$ that admits an endomorphism (actually an automorphism) $\delta_{\KK}\in \End_{\KK}(\YY)$   satisfying the $q$th cyclotomic equation
\begin{equation}
\label{cyclq}
\Phi_q(\delta_{\KK})=\sum_{i=0}^{\ell-1}\delta_{\KK}^{i \ell^{r-1}}=0
\end{equation}
in $\End_{\KK}(\YY)$. 
Let us consider the $\KK$-linear map
$$\delta_{\KK}^{*}: \Omega^1(\YY) \to \Omega^1(\YY)$$
induced by $\delta_{\KK}$. 

Suppose that $\YY$,  $\delta_{\KK}$ and  $\delta_{\KK}^{*}$ enjoy the following properties.

\begin{enumerate}
\item[(1)]
All eigenvalues of $\delta_{\KK}^{*}$ are primitive $q$th roots of unity and their multiplicities are as follows.
%If $i$ is a positive integer with
For each positive integer $i$ with
$$1 \le i < q, \ (i,\ell)=1,$$
 the multiplicity of $\zeta^{-i}$ as an eigenvalue of $\delta_{\KK}^{*}$ is $[ni/q]$. In particular,
 $\zeta^{-i}$ is an eigenvalue of $\delta_{\KK}^{*}$ if and only if $[ni/q]>0$.  
\item[(2)]
The $\Gal(\KK)$-submodule of fixed points of $\delta_{\KK}$
$$\YY^{\delta_{\KK}}:=\{y \in \YY(\bar{\KK})\mid \delta_{\KK} (y)=y\}\subset \YY(\bar{\KK})$$
is a $(n-1)$-dimensional vector space over the prime finite field $\F_{\ell}$, the image of 
$\Gal(\KK)$ in $\Aut_{\F_{\ell}}(\YY^{\delta_{\KK}})$ contains a subgroup isomorphic to  $\An$, and the corresponding $\An$-module $\YY^{\delta}$ 
is isomorphic to the natural representation of  $\An$ in
$$\left(\F_{\ell}^n\right)^0:=\{(a_1, \dots, a_n)\in \F_{\ell}^n\mid \sum_{i=1}^n a_i=0\}.$$
\end{enumerate}

Then the following assertions  are valid.

\begin{itemize}

\item[(a)]
$\delta_{\KK}^q=1_{\YY}$.
\item[(b)]
The ring homomorphism 
\begin{equation}
\label{embY}
\mathrm{i}_{\YY}: \Z[\zeta_q] \to \End_{\KK}(\YY)\subset \End(\YY)
\end{equation}
that sends $1$ to $1_{\YY}$ and $\zeta_q$ to $\delta_{\KK}$ is a ring embedding. 
\item[(c)]
$\varphi(q):=[\Q(\zeta_q):\Q]$ divides $2\dim(\YY)$.
\item[(d)]
The spectrum of the linear operator $\delta_{\KK}^{*}$ consists of more that $\varphi(q)$ distinct eigenvalues.
\item[(e)]
$$\mathrm{i}_{\YY}(\Z[\zeta_q])=\End_{\KK}(\YY)= \End(\YY),$$
i.e.,
$\End(\YY)$ coincides with its own subring generated by $\delta_{\KK}$ and the ring homomorphism $i_{\YY}$
 is a ring isomorphism. In particular, $\YY$ is absolutely simple and $\End^0(\YY) \cong \Q(\zeta_q)$.
 \end{itemize}
\end{thm}

We use Theorem \ref{generalA} in order to prove the following analogue of Theorem \ref{generalA} in finite characteristic.

\begin{thm}
\label{ordinary}
Suppose that  $\ell$ is a prime, $r$ a positive integer, and $n\ge 5$ is a positive integer that is not divisible by $\ell$. 
Let us put $q:=\ell^r$.
% is a power of a prime $\ell$ with positive integer $r$,  
Let $K$ be a perfect field of prime characteristic $p\ne \ell$ that contains a primitive $q$th root of unity, say   $\bar{\zeta}$. 
Let $Y$ be a positive-dimensional ordinary abelian variety over $K$ that admits an endomorphism (actually an automorphism) $\delta\in \End_K(Y)$   satisfying the $q$th cyclotomic equation
\begin{equation}
\label{cyclqO}
\Phi_q(\delta)=\sum_{i=0}^{\ell-1}\delta^{i \ell^{r-1}}=0
\end{equation}
in $\End_K(Y)$.  Let us consider the $K$-linear map
$$\delta^{*}: \Omega^1(Y) \to \Omega^1(Y)$$
induced by $\delta$.

Suppose that  $Y$, $\delta$ and $\delta^{*}$ enjoy the following properties.

\begin{enumerate}
\item[(1)]
All eigenvalues of $\delta^{*}$ are primitive $q$th roots of unity and their multiplicities are as follows.
If $i$ is a  positive integer with
$$1 \le i < q, \ (i,\ell)=1,$$
then the multiplicity of $\bar{\zeta}^{-i}$ as an eigenvalue of $\delta^{*}$ is $[ni/q]$. In particular,
 $\bar{\zeta}^{-i}$ is an eigenvalue of $\delta^{*}$ if and only if $[ni/q]>0$.  
\item[(2)]
The $\Gal(K)$-submodule of fixed points of $\delta$
$$Y^{\delta}:=\{y \in Y(\bar{\KK})\mid \delta (y)=y\}\subset Y(\bar{\KK})$$
is a $(n-1)$-dimensional vector space over the prime finite field $\F_{\ell}$, the image of 
$\Gal(K)$ in $\Aut_{\F_{\ell}}(Y^{\delta})$ contains a subgroup isomorphic to  $\An$, and the corresponding $\An$-module $Y^{\delta}$ is isomorphic to 
the natural representation of $\An$ in
$$\left(\F_{\ell}^n\right)^0=\{(a_1, \dots, a_n)\in \F_{\ell}^n\mid \sum_{i=1}^n a_i=0\}.$$
%of the standard permutational representation of  $\A_n$ over $\F_{\ell}$.
\end{enumerate}

Then the following assertions are valid.

\begin{itemize}
\item[(a)]
$\delta_K^q=1_Y$.
\item[(b)]
The ring homomorphism 
\begin{equation}
\label{embYB}
\mathrm{i}_Y: \Z[\zeta_q] \to \End_K(Y)\subset \End(Y)
\end{equation}
that sends $1$ to $1_Y$ and $\zeta_q$ to $\delta$ is a ring embedding. 
\item[(c)]
$\varphi(q)=[\Q(\zeta_q):\Q]$ divides $2\dim(Y)$.
\item[(d)]
The spectrum of the linear operator $\delta^{*}$ consists of more that $\varphi(q)$ distinct eigenvalues.
\item[(e)]
$$\mathrm{i}_Y(\Z[\zeta_q])=\End_K(Y)= \End(Y),$$
i.e.,
$\End(Y)$ coincides with its own subring generated by $\delta$ and the ring homomorphism $i_Y$
 is a ring isomorphism. In particular, $Y$ is absolutely simple and $\End^0(Y) \cong \Q(\zeta_q)$.
 \end{itemize}
\end{thm}

\begin{proof}[Proof of Theorem \ref{mainordinary} (modulo Theorem \ref{ordinary})]
In light of Remarks \ref{simplify}, we may and will assume that $K$ is a perfect field containing a primitive $\ell$th root of unity, $\Gal(f)=\An$
and $\ell$ does {\sl not} divide $n$.
Let us put 
$$q=\ell, \ Y=J(C_{f,\ell}), \ \delta=\delta_{\ell}.$$
Let us choose an order on the $n$-element set $
\RR_f$ that defines a bijection between $\RR_f$ and $\{1,2, \dots n\}$. This gives us a  group isomorphism  $\Gal(f)\cong\An$ and an isomorphism of the corresponding $\An$-modules $\left(\F_{\ell}^{\RR_f}\right)^{0}$ and $\left(\F_{\ell}^n\right)^0$.
Lemma \ref{fixedJ} and the assertions in the beginning of Subsection \ref{multOmega}  about eigenvalues of $\delta_{\ell}^*$ imply that
all the conditions of Theorem  \ref{ordinary} are fulfilled.
Now the desired result follows from Theorem \ref{ordinary}.
\end{proof}

\section{Endomorphism algebras of abelian varieties}
\label{endomS}

\begin{sect}{\bf Endomorphisms and torsion points.}
\label{endom}
 Let $X$ be an abelian variety of positive dimension
over an arbitrary field $K$.
 If $n$ is a positive
integer that is not divisible by $\fchar(K)$ then $X[n]$ stands for
the kernel of multiplication by $n$ in $X(\bar{K})$. It is well-known
\cite{MumfordAV} that  $X[n]$ is a free $\Z/n\Z$-module of rank
$2\dim(X)$. In particular, if $n=\ell$ is a prime then $X[\ell]$
is a $2\dim(X)$-dimensional $\F_{\ell}$-vector space.

 Since $X$ is defined over $K$,  $X[n]$ is a Galois
submodule in $X(\bar{K})$ and all points of $X[n]$ are defined over a
finite separable extension of $K$. We write
${\rho}_{n,X,K}:\Gal(K)\to \Aut_{\Z/n\Z}(X[n])$ for the
corresponding homomorphism defining the structure of the Galois
module on $X[n]$,
$${G}_{n,X,K}\subset
\Aut_{\Z/n\Z}(X[n])$$ for its image ${\rho}_{n,X,K}(\Gal(K))$
and $K(X[n])$ for the field of definition of all points of $X[n]$.
Clearly, $K(X[n])$ is a finite Galois extension of $K$ with Galois
group $\Gal(K(X[n])/K)={G}_{n,X,K}$. If $n=\ell$  then we get
a natural faithful linear representation
$${G}_{\ell,X,K}\subset \Aut_{\F_{\ell}}(X[\ell])$$
of ${G}_{\ell,X,K}$ in the $\F_{\ell}$-vector space
$X[\ell]$.

Since $X$ is defined over
$K$, one may associate with every $u \in \End(X)$ and $\sigma \in
\Gal(K)$ an endomorphism $^{\sigma}u\ \in \End(X)$ such that
$^{\sigma}u (x)=\sigma u(\sigma^{-1}x)$ for all $x \in X(\bar{K})$. We
get the  group homomorphism
\begin{equation}
\label{kappaX}
\kappa_{X}: \Gal(K) \to \Aut(\End(X));
\quad\kappa_{X}(\sigma)(u)=\ ^{\sigma}u \quad \forall \sigma \in
\Gal(K),u \in \End(X).
\end{equation} It is well-known that $\End_K(X)$
coincides with the subring of $\Gal(K)$-invariants in $\End(X)$,
i.e., $\End_K(X)=\{u\in \End(X)\mid\  ^{\sigma}u\ =u \quad \forall
\sigma \in \Gal(K)\}$. It is also well-known that $\End(X)$
(viewed as a group with respect to addition) is a free commutative
group of finite rank \cite{MumfordAV}.  In addition, $\End_K(X)$ is its {\sl pure} subgroup \cite[Sect. 4, p. 501]{ST},
i.e., the quotient $\End(X)/\End_K(X)$ is also  a free commutative
group of finite rank.
\end{sect}

\begin{sect}{\bf Abelian varieties with multiplications.}
\label{abmult} Let $E$ be a number field. Let $(X, i)$ be a pair
consisting of an abelian variety $X$ of positive dimension over
$\bar{K}$ and an embedding $i:E \hookrightarrow  \End^0(X)$. Here
We also assume that $1\in E$ goes to $1_X$.

\begin{rem}
\label{reldeg}
 It is well known  \cite[Prop. 2 on p.
36]{Shimura2}) that the degree $[E:\Q]$ divides $2\dim(X)$, i.e.
$$d_{X,E}:=\frac{2\dim(X)}{[E:\Q]}$$
is a positive integer.
\end{rem}

Let us denote by $\End^0(X,i)$ the centralizer of $i(E)$ in
$\End^0(X)$. Clearly, $i(E)$ lies in the center of the
finite-dimensional $\Q$-algebra $\End^0(X,i)$. It follows that
$\End^0(X,i)$ carries a natural structure of finite-dimensional
$E$-algebra.

Let $\OC$ be the ring of integers in $E$. If $\a$ is a non-zero
ideal in $\OC$ then the quotient $\OC/\a$ is a finite commutative
ring.
Let $\lambda$ be a maximal ideal in $\OC$. We write $k(\lambda)$
for the corresponding (finite) residue field $\OC/\lambda$ and
$\ell$ for $\fchar(k(\lambda))$.  We have
 $$\OC\supset\lambda\supset \ell\cdot\OC.$$

\begin{rem}
\label{totram} (See \cite[Remark 3.3]{ZarhinMZ2}.) Let us assume
that $\lambda$ is the only maximal ideal of $\OC$ dividing $\ell$,
i.e., $\ell\cdot\OC=\lambda^b$ where the positive integer $b$
satisfies
\newline
$[E:\Q]=b \cdot [k(\lambda):\F_{\ell}]$.

\begin{itemize}
\item[(i)] We have
 $\OC\otimes\Z_{\ell}=\OC_{\lambda}$ where $\OC_{\lambda}$ is
the completion of $\OC$ with respect to $\lambda$-adic topology. 
The ring  $\OC_{\lambda}$ is a local principal ideal domain, its
only maximal ideal is $\lambda\OC_{\lambda}$ and
$$k(\lambda)=\OC/\lambda=\OC_{\lambda}/\lambda\OC_{\lambda}, \
\ell\cdot\OC_{\lambda}=(\lambda\OC_{\lambda})^b.$$

 \item[(ii)]
 Let us choose an element $c \in \lambda$ that does not
lie in $\lambda^2$. Then
$$\lambda=\ell\cdot\OC + c\cdot\OC.$$
This implies that for all positive integers $j\le b$
$$\lambda^j=\ell\cdot\OC + c^j\cdot\OC.$$
This implies that
$$(\lambda\OC_{\lambda})^j=c^j \cdot\OC_{\lambda}.$$
In particular,
$$\ell\cdot\OC_{\lambda}=c^b\cdot\OC_{\lambda}.$$
It follows that
$$c^{-j}\ell\cdot\OC_{\lambda}=c^{b-j}\cdot\OC_{\lambda}.$$
\item[(iii)] Notice that
$E_{\lambda}=E\otimes_{\Q}\Q_{\ell}=\OC\otimes\Q_{\ell}=
\OC_{\lambda}\otimes_{\Z_{\ell}}\Q_{\ell}$ is the field coinciding
with the completion of $E$ with respect to $\lambda$-adic topology.
\end{itemize}
\end{rem}

Suppose that $X$ is defined over $K$ and $i(\OC) \subset
\End_K(X)$. Then we may view elements of $\OC$ as
$K$-endomorphisms of $X$. We write $\End(X,i)$ for the centralizer
of $i(\OC)$ in $\End(X)$ and $\End_K(X,i)$ for the centralizer of
$i(\OC)$ in $\End_K(X)$. Clearly,  
\begin{equation}
\label{EndKi}
\End_K(X,i)=\End(X,i)\bigcap \End_K(X),
\end{equation}
and $\End(X,i)$  and $\End_K(X,i)$ are pure
subgroups of $\End(X)$ and   $\End_K(X)$ respectively, i.e.,  the quotients
$\End(X)/\End(X,i)$ and $\End_K(X)/\End_K(X,i)$ are torsion-free. Recall that $\End(X)/\End_K(X)$ is also torsion-free (see the very end of Subsection \ref{endom}).
It follows from \eqref{EndKi} that  the quotient
 $\End(X,i)/\End_K(X,i)$ embeds into $\End(X)/\End_K(X)$  and therefore is also torsion-free, i.e, $\End_K(X,i)$
is a pure subgroup in $\End(X,i)$.

  We have
$$i(\OC)\subset \End_K(X,i)\subset\End(X,i)\subset \End(X).$$
It is also clear that
$$\End(X,i)=\End^0(X,i)\bigcap \End(X),$$
$$\End^0(X,i)=\End(X,i)\otimes\Q\subset
\End(X)\otimes\Q=\End^0(X).$$
 Clearly, $\End(X,i)$ carries a natural
structure of finitely generated torsion-free $\OC$-module. Since
$\OC$ is a Dedekind ring, it follows that there exist non-zero
ideals $\b_1,\ldots , \b_t$ in $\OC$ such that
$$\Z \cdot 1_X\subset \End(X,i)\cong \b_1\oplus \cdots \oplus \b_t$$
(as $\OC$-modules).

We have $\kappa_{X}(\sigma)(\End(X,i))=\End(X,i)$ for all $\sigma\in
\Gal(K)$, and
$$\End_K(X,i)=\{u\in \End(X,i)\mid\  ^{\sigma}u\ =u \quad \forall
\sigma \in \Gal(K)\},$$
where the group homomorphism $\kappa_{X}: \Gal(K) \to \Aut(\End(X))$
is defined in \eqref{kappaX}.

 Let us assume
that $\fchar(K)$ does not divide the order of  $\OC/\a$ and put
$$X[\a]:=\{x \in X(\bar{K})\mid i(e)x=0 \quad \forall e\in
\a\}.$$
 For example,  assume that $\ell \ne \fchar(K)$.  Then the order
 of $\OC/\lambda=k(\lambda)$ is a power of $\ell$ and
$X[\lambda]\subset X[\ell]$.

 Clearly, $X[\a]$ is a Galois submodule of $X(\bar{K})$. It is
also clear that $X[\a]$ carries a natural structure of
$\OC/\a$-module. (It is known \cite[Prop. 7.20]{Shimura} that this
module is free.)

\begin{rem}
\label{aLambda}
Assume in addition that $\lambda$ is the only
maximal ideal of $\OC$ dividing $\ell$
 and pick
$$c \in \lambda \setminus \lambda^2\subset \lambda \subset
 \OC.$$ By Remark \ref{totram}(ii),  $\lambda$ is
generated by $\ell$ and $c$. Hence
\begin{equation}
\label{number2}
X[\lambda]=\{x\in X_{\ell}\mid cx=0\}\subset X[\ell].
\end{equation}
More generally, for all positive integers $j \le b$ the ideal $\lambda^j$ is generated by $\ell$ and $c^j$, which implies that
\begin{equation}
\label{number2bis}
X[\lambda^j]=\{x\in X_{\ell}\mid c^jx=0\}\subset X[\ell].
\end{equation}
\end{rem}

Obviously, every endomorphism from $\End(X,i)$ leaves invariant
the subgroup $X[\a] \subset X(\bar{K})$ and induces an endomorphism
of the $\OC/\a$-module $X[\a]$. This gives rise to a natural
homomorphism
$$\End(X,i) \to \End_{\OC/\a}(X[\a]),$$
whose kernel contains $\a\cdot\End(X,i)$. Actually, the kernel
coincides with $\a\cdot\End(X,i)$, i. e., there is an embedding
\begin{equation}
\label{n3}
\End(X,i)\otimes_{\OC}\OC/\a \hookrightarrow \End_{\OC/\a}(X[\a]).
\end{equation}
See \cite[pp. 699-700]{ZarhinMZ2} for the proof.

Now we concentrate on the case of $\a=\lambda$, assuming that
$$\ell \ne \fchar(K).$$
Then $X[\lambda]$ carries the natural
structure of a $k(\lambda)$-vector space endowed with the
structure of Galois module and \eqref{n3} gives us the embedding
$$\End(X,i)\otimes_{\OC}k(\lambda)\hookrightarrow\End_{k(\lambda)}(X[\lambda])
. \eqno(4)$$ Further we will identify
$\End(X,i)\otimes_{\OC}k(\lambda)$ with its image in
$\End_{k(\lambda)}(X[\lambda])$. We write
$$\tilde{\rho}_{\lambda,X}:\Gal(K) \to
\Aut_{k(\lambda)}(X[\lambda])$$ for the corresponding
(continuous) homomorphism defining the Galois action on
$X[\lambda]$. It is known \cite[Prop. 7.20] {Shimura} (see also
\cite{Ribet2}) that
\begin{equation}
\label{dimXlambda}
\dim_{k(\lambda)}X[\lambda]= \frac{2\dim(X)}{[E:\Q]}:=d_{X,E}.
\end{equation}

 Let us put
$$\tilde{G}_{\lambda,X}=\tilde{G}_{\lambda,i,X}:=\tilde{\rho}_{\lambda,X}(\Gal(K)) \subset
\Aut_{k(\lambda)}(X[\lambda]).$$ Clearly, $\tilde{G}_{\lambda,X}$
coincides with the Galois group of the field extension
$K(X[\lambda])/K$, where $K(X[\lambda])$ is the field of
definition of all points in $X[\lambda]$.

It is also clear that the image of $\End_K(X,i)
\otimes_{\OC}k(\lambda)$ lies in the centralizer
$\End_{\tilde{G}_{\lambda,i,X}}(X[\lambda])$ of $\tilde{G}_{\lambda,i,X}$ in
$\End_{k(\lambda)}(X[\lambda])$.

\begin{lem}
\label{overK} (See \cite[Lemma 3.8]{ZarhinMZ2}.) Suppose that
$i(\OC)\subset \End_K(X)$. If $\lambda$ is a maximal ideal in $\OC$
with $\ell \ne \fchar(K)$, and
  $\End_{\tilde{G}_{\lambda,i,K}}(X[\lambda])=k(\lambda)$
then $\End_K(X,i)=\OC$.
\end{lem}
\end{sect}

\section{Abelian schemes}
\label{Aschemes}

Let $\KK$ be a complete discrete valuation field with discrete valuation ring $O_{\KK}$
with maximal ideal $\mm$
and perfect residue field $\kappa:=O_{\KK}/\mm$.  We write 
$$O_{\KK} \to O_{\KK}/\mm=\kappa, \ a \mapsto \bar{a}=a+\mm$$
for the residue map.  Suppose that $q$ is a positive integer that is {\sl not} divisible by $\fchar(\kappa)$ and such that $\kappa$ contains a primitive $q$th root of unity.
 Then it follows from Hensel's Lemma that $O_{\KK}$ contains a primitive $q$th root of unity, and
the reduction map defines a canonical isomorphim
$$\mu_{q,\KK} \cong \mu_{q,\kappa}, \ \gamma \mapsto \bar{\gamma}$$
between the order $q$ cyclic groups of $q$th roots of unity 
$$\mu_{q,\KK}\subset O_{\KK}^{*}\subset \KK^{*}$$ and
$ \mu_{q,\kappa}\subset \kappa^{*}$.

Since $\KK$ is complete, the absolute Galois group $\Gal(\KK)=\Aut(\bar{\KK}/\KK)$ coincides with the {\sl decomposition group}, and the inertia subgroup $\I$ of $\Gal(\KK)$ is a closed normal subgroup of $\Gal(\KK)$.  The natural extension of the reduction map  to $\bar{\KK}$ induces the surjective continuous group homomorphism \cite[Sect. 1]{ST}
\begin{equation}
\label{redG}
\mathrm{red}_{\Gal}: \Gal(\KK)/\I\to \Gal(\kappa),
\end{equation}
which is actually an {\sl isomorphism} of compact (profinite) groups.

We write $S$ for $\Spec(O_{\KK})$.
It is well known that $S$ consists of the {\sl generic point} $\eta$ that corresponds to $\{0\}$
and the closed point $s$ that corresponds to $\mm$.  

Let $f:\XX \to S$ be an abelian scheme over $S$
of positive relative dimension $g$. We write $X_{\KK}$ for its generic fiber, which is a $g$-dimensional abelian variety over $\KK$,
and $X_s$ for its closed fiber, which is a $g$-dimensional abelian variety over $\kappa$.
By definition, $\XX$ is a separated scheme; it is known \cite[Remark 1.2 on p. 1]{FC} that $\XX$ is a scheme of finite type over $S$ and, therefore, is {\sl noetherian}.
In addition, $\XX$ is a {\sl N\'eron model} of $X_{\KK}$ \cite[Prop. 8 on p. 15]{Neron}. In particular, the natural ring homomorphism
\begin{equation}
\label{ringK}
\End_S(\XX) \to \End_{\KK}(X_{\KK}),  \ \mathfrak{u} \mapsto \mathfrak{u}_{\KK}
\end{equation}
is a {\sl  isomorphism}. Notice that the natural ring homomorphism
\begin{equation}
\label{rings}
\End_S(\XX) \to \End_{\kappa}(X_{s}),  \ \mathfrak{u} \mapsto \mathfrak{u}_{s}
\end{equation}
is an {\sl  embedding}, thanks to Rigidity Lemma \cite[Ch. 6, Sect. 1, Prop. 6.1 and Cor. 6.2]{MumfordGIT}.

As usual, we write $\OC_{\XX}$ for the structure sheaf of $\XX$. Clearly, the ring $\Gamma(\XX,\OC_{\XX})$ of global sections of $\XX$ is
a $O_{\KK}$-algebra.  Since $f$ is proper, the  $O_{\KK}$-module $\Gamma(\XX,\OC_{\XX})$ is finitely generated.

The following assertion is, actually,  contained in \cite[Theorem 1 on p. p. 493 and Lemma 2 on p. 495]{ST}.

\begin{lem}
\label{torsionGalois}
Let $n$ be a positive integer that is {\sl not} divisible by $\fchar(\kappa)$. 
\begin{itemize}
\item[(i)]
The action of the inertia subgroup $\I$ on the $\Gal(K)$-module ${X_{\KK}}[n]$ is trivial and, therefore, one may view ${X_{\KK}}[n]$  as the $\Gal(\kappa)$-module via \eqref{redG}. 
\item[(ii)]
There is a canonical isomorphism of  $\Gal(\kappa)$-modules 
\begin{equation}
\mathrm{red}_n: X_{\KK}[n] \cong X_{s}[n]
\end{equation}
such that the diagram
\begin{equation}\label{diagram1}
\begin{CD}
  X_{\KK}[n]  @>{\mathfrak{u}_{\KK}}>> X_{\KK}[n] \\
@V{\mathrm{red}_n}VV@V{\mathrm{red}_n}VV\\
X_{s}[n]@>{\mathfrak{u}_s}>>X_{s}[n]
\end{CD}\end{equation}

is commutative for all $\mathfrak{u}\in \End_S(\XX)$.
\item[(iii)]
The $\Gal(\kappa)$-modules of fixed points 
$$X_{\KK}[n]^{\mathfrak{u}_{\KK}}=\{x \in X_{\KK}[n]\mid \mathfrak{u}_{\KK}(x)=x\}$$
and
$$X_{s}[n]^{\mathfrak{u}_s}=\{x \in X_{s}[n]\mid \mathfrak{u}_s(x)=x\}$$
are isomorphic for all $\mathfrak{u}\in \End_S(\XX)$.
\end{itemize}
\end{lem}

\begin{proof}
(i) follows from \cite[Theorem 1 on p. p. 4 and Lemma 2 on p. 495]{ST}.
This implies that $X_{\KK}[n]^{\I}=X_{\KK}[n]$.

The desired isomorphism of Galois modules 
$$\mathrm{red}_n: X_{\KK}[n]^{\I}=X_{\KK}[n]\cong X_{s}[n]$$
 is constructed
in  the proof of Lemma 2 on p. 495 of \cite{ST}. It follows readily
from the construction that all the  diagrams \eqref{diagram1} are commutative.
This proves (ii), which immediately implies (iii).
\end{proof}

The following assertion is probably well known but I failed to find a suitable reference.

\begin{lem}
\label{basic}
The  scheme $\XX$ is  normal reduced irreducible and $\Gamma(\XX,\OC_{\XX})=O_{\KK}$.
\end{lem}

\begin{proof}
The abelian scheme $\XX$ is automatically of finite presentation over $S$ \cite[Remark 1.2(a) on p. 3]{FC}.  
Since $S$ is normal and $f: \XX \to S$ is smooth, it follows from  \cite[Prop. 5.5 on pp. 235--236]{MO} that $\XX$ is {\sl normal} and therefore {\sl reduced}.  It follows from \cite[Criterion 5.5.2 on pp. 234--235]{MO} that  the {\sl irreducible components} of the topological space $\XX$ are {\sl disjoint}.  Since $X_{\KK}$ is a nonempty open irreducible subset of $\XX$, it lies in an irreducible component of $\XX$, say  $Y_1$.   Since $X_{s}$ is a nonempty closed irreducible subset of $\XX$, it lies in an irreducible component of $\XX$, say  $Y_2$. This implies that the topological 
space $\XX$ is a union of $Y_1$ and $Y_2$.  Since irreducible components are closed subsets, both $Y_1$ and $Y_2$ are closed subsets in $\XX$. Since $f$ is proper, $f(Y_1)$ is a closed subset of $S$. By definition, $f(Y_1)$ contains the generic point $\eta$ of $S$ and therefore contains its closure $S$.  This implies that $f(Y_1)=S$.  Hence
 $Y_1$ meets $X_s$ and therefore meets $Y_2$ as well. It follows that $Y_1=Y_2$ and therefore $\XX=Y_1$ is {\sl irreducible}.  This implies that $\Gamma(\XX,\OC_{\XX})$ is an $O_{\KK}$-algebra without zero divisors. On the other hand, the properness of $f$ implies that the $O_{\KK}$-algebra $\Gamma(\XX,\OC_{\XX})$ is a finitely generated $\Gamma(S,\OC_S)=O_{\KK}$-module.

Recall that $X_{\KK}$ is an open nonempty subset of $\XX$. The irreducibility of $\XX$ implies that
$X_{\KK}$ is dense in $\XX$. It follows that the restriction map
$$\Gamma(\XX,\OC_{\XX}) \to \Gamma(X_{\KK},\OC_{X_{\KK}})=\KK$$
is {\sl injective}.  This implies that the (isomorphic) image of $\Gamma(\XX,\OC_{\XX})$ in $\KK$
is an {\sl order} containing $O_{\KK}$ and therefore coincides with  $O_{\KK}$. It follows that
$\Gamma(\XX,\OC_{\XX})=O_{\KK}$.
\end{proof}

We write $\Omega^1(X_{\KK})$ and $\Omega^1(X_{s})$ for the  $\KK$-vector space $\Omega^1(X_{\KK})$ of differentials of the first kind on $X_{\KK}$ and the $\kappa$-vector space $\Omega^1(X_{s})$ of differentials of the first kind on $X_{s}$ respectively (both of dimension $g$).

\begin{thm}
\label{specOmega}
Let $q$ be a positive integer  that is not divisible by $\fchar(\kappa)$ such that $\kappa$ contains a primitive $q$th root of unity.

 Suppose that $\delta_S$ is an automorphism of the abelian scheme $\XX/S$ such that $\delta_S^N$ is the identity map.
Let $\delta_{\KK}: X_{\KK} \to X_{\KK}$ and $\delta_s:X_s \to X_s$ be the automorphisms of abelian varieties $X_{\KK}$ and $X_s$ induced by $\delta_S$. Let us consider the linear automorphisms
$$\delta_{\KK}^{*}: \Omega^1(X_{\KK}) \to \Omega^1(X_{\KK}), \ 
\delta_{\kappa}^{*}: \Omega^1(X_{s}) \to \Omega^1(X_{s})$$
%of the $\KK$-vector space $\Omega^1(X_{\KK})$ of differentials of the first kind on $X_{\KK}$ %and the $\kappa$-vector space $\Omega^1(X_{s})$ of differentials of the first kind on $X_{s}$
 induced by $\delta_{\KK}$ and $\delta_s$ respectively. 

Let $\spec(\delta_{\KK}^{*})\subset \bar{\KK}$ and  $\spec(\delta_{\kappa}^{*})\subset \bar{\kappa}$
be the sets of eigenvalues of $\delta_{\KK}^{*}$ and $\delta_{\kappa}^{*}$ respectively.

Then:

\begin{itemize}
\item[(i)]
Both $\delta_{\KK}^{*}$ and $\delta_{\kappa}^{*}$ are diagonalizable linear operators and
$$\spec(\delta_{\KK}^{*})\subset \mu_{q,\KK}\subset \KK, \ \spec(\delta_{\kappa}^{*})\subset \mu_{q,\kappa}\subset \kappa.$$
\item[(ii)]
A root of unity  $\gamma\in \mu_{q,\KK}$ lies in $\spec(\delta_{\KK}^{*})$ if and only if $\bar{\gamma}$ lies in $\spec(\delta_{\kappa}^{*})$.
\item[(iii)]
The multiplicity of each $\gamma \in \spec(\delta_{\KK}^{*})$ coincides with the multiplicity of
$\bar{\gamma} \in \spec(\delta_{\kappa}^{*})$.
\end{itemize}
\end{thm}

\begin{proof}
Clearly, 
$$(\delta_{\KK}^{*})^q=1_{X_{\KK}},  (\delta_{\kappa})^N=1_{X_s}$$
and therefore both $(\delta_{\KK}^{*})^N$ and $(\delta_{\kappa}^{*})^N$ are the identity automomorphisms of $\Omega^1(X_{\KK}) $ and $\Omega^1(X_{s})$ respectively.
The conditions on $q$ imply that both $\delta_{\KK}^{*}$ and $\delta_{\kappa}^{*}$ are diagonalizable and their eigenvalues lie in $\mu_{q,\KK}$ and $\mu_{q,\kappa}$ respectively. This proves (i).

In order to prove (ii) and (iii), recall \cite[Cor. 3 on p. 102]{Neron} that the sheaf $\Omega^{1}_{\XX/S}$ of relative differentials is a free $\OC_{\XX}$-module of rank $g$ and therefore is generated by certain $g$ sections that remain linearly independent at every point of $\XX$. It follows from Lemma \ref{basic} that the group $\Gamma(\XX, \Omega^{1}_{\XX/S})$ of its global sections is a free $O_{\KK}$-module of rank $g$.  Let 
$$\delta_S^{*}: \Gamma(\XX, \Omega^{1}_{\XX/S}) \to \Gamma(\XX, \Omega^{1}_{\XX/S})$$
be the automorphism of $\Gamma(S,\OC_S)=\OC_K$-module $\Gamma(\XX, \Omega^{1}_{\XX/S})$ induced by $\delta_S$.  Clearly, $(\delta_S^{*})^N$ is the identity map. Our conditions on $N$ imply the existence of a basis $\{\omega_1, \dots, \omega_g\}$ of  $\Gamma(\XX, \Omega^{1}_{\XX/S})$ such that
for each $i=1, \dots g$ there exists $\gamma_i\in \mu_{N,\KK}$ such that
$$\delta_S^{*} \omega_i=\gamma_i \omega_i.$$
Clearly, $\{\omega_1, \dots, \omega_g\}$ remain linearly independent at every point of $\XX$. This implies that
their {\sl restrictions} $\{\omega_{1,s}, \dots, \omega_{g,s}\}$ to $X_s$ constitute a basis of the $\kappa$-vector space
$$\Gamma(X_s,  \Omega^{1}_{X/\kappa})=\Omega^1(X_{s}),$$ their {\sl restrictions}
$\{\omega_{1,\KK}, \dots, \omega_{g,\KK}\}$ to $X_{\KK}$ constitute a basis of the $\KK$-vector space
$$\Gamma(X_{\KK},  \Omega^{1}_{X_{\KK}/\KK})=\Omega^1(X_{\KK}).$$
 The functoriality of the formation of relative differentials and its compatibility with base change   \cite[Prop. 3 on p. 34 and p. 35]{Neron} imply that
$$\delta_{\KK}^{*} \omega_{i,\KK}=\gamma_i \cdot \omega_{i,\KK}, \ \delta_s^{*} \omega_{i,s}=\bar{\gamma_i}\cdot \omega_{i,s}  .$$
for all $i=1, \dots g$. This implies readily (ii) and (iii).
\end{proof}

\begin{thm}
\label{GaloisDelta}
We keep the notation and assumptions of Theorem \ref{specOmega}.
Suppose that there exist a prime $\ell$ and a positive integer $r$ such that
$q=\ell^r$ (in particular, $\ell \ne \fchar(\kappa)$, the field $\kappa$ contains a primitive $q$th root of unity) and $\delta_S \in \End_S(\XX)$ satisfies the $q$th cyclotomic equation
$$\Phi_q(\delta)=0,$$
where the $q$th cyclotomic polynomial
$$\Phi_q(t)=\frac{t^{\ell^r}-1}{t^{\ell^{r-1}}-1}=\sum_{j=0}^{l-1}t^{j \ell^{r-1}}\in \Z[t].$$
Let us consider the subgroups of fixed points
$$X_{\KK}^{\delta_{\KK}}=\{x \in X_{\KK}(\bar{\KK})\mid \delta_{\KK}(x)=x\}\subset X_{\KK}(\bar{\KK})$$
 and
$$X_{s}^{\delta_{s}}=\{y \in X_{s}(\kappa)\mid \delta_{s}(y)=y\}\subset X_{s}(\bar{\kappa}).$$

Then:
\begin{itemize}
\item[(i)]
$X_{\KK}^{\delta_{\KK}}$ is a $\Gal(K)$-submodule of $X_{\KK}[\ell]$  and
$X_{s}^{\delta_{s}}$ is a $\Gal(\kappa)$-submodule of $X_s[\ell]$.
\item[(ii)]
$X_{s}^{\delta_{s}}$ viewed as the  $\Gal(K)$-module is isomorphic to 
$X_{\KK}^{\delta_{\KK}}$.
\end{itemize}
\end{thm}

\begin{proof}
(i) We have 
$$\ell x=\sum _{j=0}^{l-1}\delta_{\KK}^{j \ell^{r-1}}(x)=0, \ \ell y=\sum _{j=0}^{l-1}\delta_{s}^{j \ell^{r-1}}(y)=0.$$

(ii) It follows readily from Lemma \ref{torsionGalois} applied to $n=\ell$ and $\mathfrak{u}=\delta$.
\end{proof}

%\section{Canonical liftings}
\begin{rem}
\label{ordinarylift}
Let $\kappa$ be a perfect field of prime characteristic $p$, $W(\kappa)$ its ring of Witt vectors and $\mathbf{K}$ the field of fractions of $W(\kappa)$. The field $\mathbf{K}$ is a complete discrete valuation field of characteristic zero with valuation ring  $W(\kappa)$ and residue field $\kappa$.

 Let $X_0$ be an abelian variety over $\kappa$ of positive dimension $g$. If $X_0$ is {\sl ordinary} then there exists a {\sl canonical} Serre-Tate lifting of $X_0$ - an abelian scheme $\mathcal{X}$ over $S=\Spec(W(\kappa))$, whose closed fiber coincides with $X_0$, and the natural ring homomorphism  $\End_S(\mathcal{X}) \to \End_{\kappa}(X_0)$ is an isomorphism \cite[p. 172, Th. 3.3]{Messing}. We write $X_{\mathbf{K}}$ for the generic fiber of $\mathcal{X}$; it is a $g$-dimensional abelian variety over ${\mathbf{K}}$. The abelian scheme $\mathcal{X}$ is a N\'eron model of $X_{\mathbf{K}}$ and therefore the natural ring homomorphism  $\End_S(\mathcal{X}) \to \End_{K}(X_{\mathbf{K}})$ is an isomorphism \cite{Neron}.
\end{rem}

\section{Proof of Theorem \ref{generalA}}
\label{pfGA}
Since $\Phi_q(t)$ divides $t^q-1$ in $\Z[t]$, it follows from \eqref{cyclq} that $\delta_{\KK}^q-1_{\YY}=0$, which proves (a).
The ring homomorphism $\mathrm{i}_{\YY}$ defined in 
\eqref{embY} is an embedding, because the $q$th cyclotomic polynomial is irreducible over $\Q$.  This proves 
(b), which, in turn, implies (c) in light of Remark \ref{reldeg}. 
The assertion (d) about the cardinality of the spectrum of $\delta_{\KK}^{*}$ is actually 
proven on pp. 357--358 of \cite{ZarhinM}.

{\bf Proof} of (e).
Let us put
$$E=\Q(\zeta_q), \ O=\Z[\zeta_q]\subset \Q(\zeta_q)=E$$
and write $\lambda$ for the maximal ideal $(1-\zeta_q)O$ of $O$. It is well known that 
\begin{itemize}
\item
$\lambda$ is the only maximal ideal of $O$ that lies above $\ell$;
\item
 the residue field
$k(\lambda)=O/\lambda=\F_{\ell}$.
\end{itemize}
As in \cite{ZarhinM} we define
$$\YY_{\lambda}=\{y\in \YY(\bar{\KK})\mid \mathrm{i}_{\YY}(e)y=0 \ \forall e \in \lambda\}.$$
Since $\lambda$ is generated by $(1-\zeta_q)O$ and $ \mathrm{i}_{\YY}(\zeta_q)=\delta_{\KK}$,
$$\YY_{\lambda}=\{y\in \YY(\bar{\KK})\mid (1_{\YY}-\delta_{\KK})y=0\}=\YY^{\delta}.$$
Combining Property (2)  of Theorem \ref{generalA} and \eqref{dimXlambda}, we obtain that
\begin{equation}
\label{dimYE}
n-1=\frac{2\dim(\YY)}{[E:\Q]}, \ 2 \dim(\YY)=(n-1)\varphi(q).
\end{equation}

Let $\Q[\delta_{\KK}]$ be the $\Q$-subalgebra   of $\End^0(\YY)$ generated by $\delta_{\KK}$. Actually, 
 $\Q[\delta_{\KK}]$ is a (sub)field isomorphic to $\Q(\zeta_q)$. We write $\End^0(\YY,\mathrm{i}_{\YY})$  for the {\sl centralizer}  of $\Q[\delta_{\KK}]$ in $\End^0(\YY)$. Clearly
 $$\delta_{\KK} \in \Q[\delta_{\KK}] \subset  \End^0(\YY,\mathrm{i}_{\YY})\subset \End^0(\YY).$$

It follows from Property (2) and Theorem 4.7 of \cite{ZarhinCrelle} that the $\Gal(\KK)$-module $\YY_{\lambda}$ is {\sl very simple} in a sense of \cite{ZarhinTexel}. It follows from Theorem 3.8 of \cite{ZarhinM} that one of the following conditions holds.

\begin{itemize}
\item[(A)]
The subring $\mathrm{i}_{\YY}(O)$ coincides with its own centralizer in $\End(\YY)$.
In other words, every endomorphism in $\End(\YY)$ that commutes with $\delta_{\KK}$ may be presented as a polynomial in $\delta_{\KK}$ with integer coefficients. Equivalently, 
$$\Q[\delta_{\KK}]=\End^0(\YY,\mathrm{i}_{\YY}).$$
\item[(B)]
 $\End^0(\YY,\mathrm{i}_{\YY})$   is a {\sl central simple} $\Q[\delta_{\KK}]$-algebra of dimension 
$\left(2\dim(\YY)/[E:\Q]\right)^2$. (It follows from \eqref{dimYE} that the $\Q[\delta_{\KK}]$-dimension of $\End^0(\YY,\mathrm{i}_{\YY})$ is $(n-1)^2$.) 
\end{itemize}

Let us prove that the case (B) does {\sl not} occur.
Every eigenspace of $\delta_{\KK}^{*}$ is invariant under the natural action of $\End^0(\YY,\mathrm{i}_{\YY})$.  The  properties of $\End^0(\YY,\mathrm{i}_{\YY})$ imply that the multiplicity of every eigenvalue of $\delta_{\KK}^{*}$ is divisible by $(n-1)$. (It is the only place where we use that $\fchar(\KK)=0$.) Notice that for each positive integer $i<q$ with $(i,\ell)=1$
\begin{equation}
\label{iqmi}
\left[\frac{ni}{q}\right] +\left[\frac{n(q-i)}{q}\right] =n-1.
\end{equation}
This implies that if $\zeta^{-i}$ is an eigenvalue of $\delta_{\KK}^{*}$ then its multiplicity $[ni/q]$ is a positive integer divisible by $(n-1)$ and therefore equals $(n-1)$. By \eqref{iqmi}, 
$[n(q-i)/q]=0$, i.e.,  $\zeta^{-(q-i)}=\zeta^{-i}$ is {\sl not} an eigenvalue of $\delta_{\KK}^{*}$.
This implies that the number of distinct  eigenvalues of $\delta^{*}$ does {\sl not} exceed $\varphi(q)/2$, which is not true. This proves that the case (B) does {\sl not} occur.

So, the case (A) holds, i.e., the field $ \Q[\delta_{\KK}]\cong \Q(\zeta_q)$ coincides with its own centralizer in $\End^0(\YY)$. Then $ \Q[\delta_{\KK}]$ contains the center $\mathfrak{C}_{\YY}$ of $\End^0(\YY)$.  It follows from Corollary 2.2 of \cite{ZarhinM} that $\Q[\delta_{\KK}]$ coincides with $\mathfrak{C}_{\YY}$. This implies that the centralizer of $ \Q[\delta_{\KK}]$ is the whole $\End^0(\YY)$ and, therefore, $\End^0(\YY)=\Q[\delta_{\KK}]$. Since $\Z[\delta_{\KK}]\cong \Z[\zeta_q]$ and $\Z[\zeta_q]$ is the maximal order in $\Q[\zeta_q]$,  
$$\End(\YY)=\Z[\delta_{\KK}]\cong \Z[\zeta_q].$$
This ends the proof of Theorem \ref{generalA}.

\section{Proof of Theorem \ref{ordinary}}
\label{Pordinary}
Let $W(K)$ be the ring of Witt vectors over $K$ and $\KK$ its field of fractions, which has characteristic $0$. Since $K$ contains a primitive $q$th root of unity,
there is a primitive $q$th root of unity
$$\zeta \in W(K)\subset  \KK,$$
that goes to $\zeta$ in $K$ under the reduction map.

 Let  $S:=\Spec (W(K))$ and $f:\mathcal{Y}\to S$ be the canonical Serre-Tate lifting of $Y$ \cite{Messing}, which is an abelian scheme over $S$, whose closed fiber is $Y$. We write $\YY$ for the generic fiber of $f$, which is an abelian variety over $\KK$ with the same dimension as $Y$.  We know (see Remark \ref{ordinarylift}) that the natural ring isomorphisms
\begin{equation}
\label{liftHom}
                    \End_{K}(Y)     \leftarrow           \End_S(\mathcal{Y})\to \End_{\KK}(\YY)
\end{equation}
are isomorphisms. Let $\delta_S\in \End_S(\mathcal{Y})$ be the preimage (lifting) of $\delta$ in 
$\End_S(\mathcal{Y})$ and $\delta_{\KK}\in  \End_{\KK}(\YY)$ be the image (restriction to the generic fiber) of $\delta_S$ in $\End_{\KK}(\YY)$. 
Isomorphisms \eqref{liftHom} induce ring isomorphisms
\begin{equation}
\label{liftDelta}
\Z[\delta_{\KK}]  \leftarrow \Z[\delta_S] \to \Z[\delta].
\end{equation}

In light of \eqref{ordinary},
\begin{equation}
\label{PhiD}
\Phi_q(\delta_S)=0,  \ \Phi_q(\delta_{\KK})=0; \
\delta_S^q=1_{\mathcal{Y}}, \ \delta_{\KK}^q=1_{\YY}.
\end{equation}
%It follows from Theorem \refheorem \ref{generalA}(i) applied to $\YY/K$ that
Then the ring homomorphism
$$i_{\YY}: \Z[\zeta_q] \to \Z[\delta_{\KK}]\subset \End_{\KK}(\YY) $$  
that sends $1$ to $1_{\YY}$ and $\zeta_q$ to $\delta_{\KK}$ is a ring embedding.
This proves  Theorem \ref{ordinary}(b), which, in turn, implies Theorem \ref{ordinary}(b). In order to prove Theorem \ref{ordinary}(a,b), we are going to apply Theorem  \ref{generalA}  to $\YY/\KK$.  Theorems \ref{GaloisDelta} and \ref{specOmega} applied to $\mathcal{Y}/S$ imply that all the conditions of  Theorem  \ref{generalA} are fulfilled. Applying Theorem  \ref{generalA}(d,e) to $\YY/\KK, \delta_{\KK}, \zeta$, we conclude that the spectrum of $\delta^{*}$ consists of more than $\varphi(q)/2$ eigenvalues and 
$$i_{\YY}(\Z[\zeta_q]) = \End_{\KK}(\YY) =\End(\YY).$$  
It remains to recall that there is a canonical ring  isomorphism between $\End(Y)$ and $\End(\YY)$
under which $\delta=i_Y(\zeta_q)$ goes to $\delta_{\KK}=i_{\YY}(\zeta_q)$.

\section{Superelliptic Jacobians}
\label{superQ}
We are heading for the proof of Theorem \ref{mainordinaryq}.
\label{proofm} We keep the notation of Section \ref{one}. 
In
particular, $\ell$ is a prime, $K$ is a field of characteristic $p \ne \ell$ that contains
$\zeta$, which is a primitive $q=\ell^r$th root of unity; $f(x)\in K[x]$ is a separable 
polynomial of degree $n\ge 4$.

 We write $C_{f,q}$
for the superelliptic $K$-curve $y^{q}=f(x)$ and $J(C_{f,q})$ for
its Jacobian. Recall that $J(C_{f,q})$ is an abelian variety that is
defined over $K$. 
In \cite{ZarhinM} I constructed an abelian positive-dimensional $K$-subvariety $J^{(f,q)}\subset J(C_{f,q})$
that enjoys the following properies.

\begin{itemize}
\item[(i)]
There exists a $K$-isogeny  of abelian varieties
$$\prod_{j=1}^r J^{(f,\ell^i)}\to J(C_{f,q}).$$
\item[(ii)]
There is a ring embedding
$$e_q: \Z[\zeta_q]\hookrightarrow \End_K(J^{(f,q)})$$
that sends $1$ to $1_{J^{(f,q)}}$ and enjoys the following properies.
\begin{enumerate}

\item[(1)]
$\delta:=e_q(\zeta_q)\in \End_K(J^{(f,q)})$ is an automorphism of $J^{(f,q)}$ that satisfies
$$\sum_{i=0}^{\ell-1}\delta^{i \ell^{r-1}}=0, \ \delta^q=1_{J^{(f,q)}}.$$
\item[(2)]
If $l \nmid n$ then the $\Gal(K)$-(sub)module $\left(J^{(f,q)}\right)^{\delta}$ of $\delta$-invariants  is isomorphic to $\left(\F_{\ell}^{\RR_f}\right)^0$.

\item[(3)]
If $l \nmid n$ then the eigenvalues of $\delta^{*}:\Omega^1(J^({f,q)})$ are primitive $q$th roots of unity.
In addition, for each positive integer $i$ with
$$1 \le i < q, \ (i,\ell)=1$$
the multiplicity of $\zeta^{-i}$ as an eigenvalue of $\delta^{*}$ is $[ni/q]$. In particular,
 $\zeta^{-i}$ is an eigenvalue of $\delta^{*}$ if and only if $[ni/q]>0$. 
 \end{enumerate} 
\end{itemize}

\begin{proof}[Proof of Theorem \ref{mainordinaryq}]
In light of Remarks \ref{simplify}, we may assume that $\ell$ does not divide $n$, the field $K$ is perfect and  contains a primitive $q$th root of unity, and $\Gal(f)=\An$. Since $J(C_{f,q})$ is ordinary, it follows from Property (i) that  $J^{(f,\ell^j)}$ is ordinary if $1 \le j\le r$.  Applying Theorem \ref{ordinary} to
$$Y=J^{(f,\ell^j)}, \delta=e_{\ell^j}(\zeta_{\ell^j})$$
for all such $j$, we obtain that
$$\End(J^{(f,\ell^j)})\cong \Z[\zeta_{\ell^j}], \ \End^0(J^{(f,\ell^j)})\cong \Q(\zeta_{\ell^j}).$$
In particular, all $J^{(f,\ell^j)}$ are absolutely simple and their endomorphism algebras are non-isomorphic for distinct $j$.
Now Property (i) implies that
$$\End^0(J(C_{f,q}) )\cong \oplus_{j=1}^r \End^0(J^{(f,\ell^j)})\cong  \oplus_{j=1}^r \Q(\zeta_{\ell^j}).$$
This implies that
$$\End^0(J(C_{f,q})) \cong \oplus_{j=1}^r \Q(\zeta_{\ell^j}).$$
On the other hand,  $\End^0(J(C_{f,q}))$ contains the $\Q$-subalgebra $\Q[\delta_q]$ and the latter is isomorphic to $\oplus_{j=1}^r \Q(\zeta_{\ell^j})$
(see  Subsect. \ref{superqdelta}). Now $\Q$-dimension arguments imply that
$\End^0(J(C_{f,q})=\Q[\delta_q]$.
\end{proof}

\section{Corrigendum to \cite{ZarhinGanita}}

\begin{itemize}
\item
Page 515,  line 6: one should read $(x,\zeta_q y)$ (not $(x,\zeta y)$ ).
\item
Page 515, line 9: one should read $\delta_q$ (not $\delta_p$ ).
\item
Page 515, line 18 (displayed formula): one should read
$(J(C_{f,q}))$ (not $((C_{f,q}))$ ).
\item
Page 515, line 6 from below (displayed formula): one should read $J^{(f,q)}$ (not $J^{f,q}$ ).
\item
Page 515, line 4 from below(displayed formula)  starts with $i$ (not $1$ ).
\item
Page 516, line 5 (displayed formula): one should read $J^{(f,q)}$ (not $J^{f,q}$ ).
\item
Page 516, line 11: one should read $\Z[\zeta_q]$ (not $Z[\zeta_q]$ ).
\item
Page 516,  line 13: one should read $J^{(f,q)}_{\lambda}$ (not $J^{f,q}_{\lambda}$ ).
\end{itemize}

\end{document}